\documentclass[11pt,a4paper]{article}

\usepackage{amsmath}
\usepackage{amsthm}
\usepackage{amsfonts}
\usepackage{color}
\usepackage{setspace}
\usepackage{mathrsfs}
\usepackage{graphicx}

\usepackage{calc}
\usepackage{tikz}

\newtheorem{theorem}{Theorem}

\newtheorem{lemma}{Lemma}
\newtheorem{problem}{Problem}

\newtheorem{observation}{Observation}

\newtheorem{claim}{Claim}
\newtheorem{corollary}{Corollary}

\usetikzlibrary{decorations.markings}
\tikzstyle{vertex}=[circle, inner sep=1.2pt, minimum size=3pt]

\tikzstyle{filledvertex}=[circle, draw, fill, inner sep=1.2pt, minimum size=3pt]
\newcommand{\vertex}{\node[vertex]}

\tikzstyle{directed}=[postaction={decorate,
	decoration={markings,mark=at position 0.5 with {\arrow{stealth}}}
}]

\usepackage{geometry}
\begin{document}
\begin{spacing}{1.2}

\title{Kernels by rainbow paths in arc-colored tournaments \thanks{The first author is supported by NSFC (No. 11601430) and China Postdoctoral Science Foundation (No. 2016M590969);
the second author is supported by  NSFC (No. 11601429);
and the third author is supported by NSFC (Nos. 11571135 and 11671320).}}

\author{\quad Yandong Bai \thanks{Corresponding author. E-mail addresses: bai@nwpu.edu.cn (Y. Bai), binlongli@nwpu.edu.cn (B. Li), sgzhang@nwpu.edu.cn (S. Zhang).},
\quad Binlong Li,
\quad Shenggui Zhang\\[2mm]
\small Department of Applied Mathematics, Northwestern Polytechnical University, \\
\small Xi'an 710129, China}

\date{\today}
\maketitle

\begin{abstract}
For an arc-colored digraph $D$,
define its {\em kernel by rainbow paths} to be a set $S$ of vertices such that
(i) no two vertices of $S$ are connected by a rainbow path in $D$,
and (ii) every vertex outside $S$ can reach $S$ by a rainbow path in $D$.
In this paper,
we show that it is NP-complete to decide whether an arc-colored tournament has a kernel by rainbow paths,
where a {\em tournament} is an orientation of a complete graph.
In addition,
we show that every arc-colored $n$-vertex tournament with
all its strongly connected $k$-vertex subtournaments, $3\leq k\leq n$, colored with at least $k-1$ colors has a kernel by rainbow paths,
and the number of colors required cannot be reduced.

\smallskip
\end{abstract}

\section{Introduction}

All graphs (digraphs) considered in this paper are finite and simple,
i.e., without loops or multiple edges (arcs).
For terminology and notation not defined here,
we refer the reader to Bang-Jensen and Gutin \cite{BG2008}.

A path (cycle) in a digraph always means a {\em directed} path (cycle)
and a {\em $k$-cycle $C_{k}$} means a cycle of length $k$, where $k\geq 2$ is an integer.
Let $D$ be a digraph and $m$ a positive integer.
Call $D$ an $m$-colored digraph if its arcs are colored with at most $m$ colors.
For an arc $uv$,
denote by $c(uv)$ the color assigned to $uv$.
An arc-colored digraph is {\em monochromatic} if all arcs receive the same color;
it is {\em properly colored} if any two consecutive arcs receive distinct colors;
and it is {\em rainbow} if any two arcs receive distinct colors.

For a digraph $D$,
define its {\em kernel} to be a set $S$ of vertices of $D$ such that
(i) no two vertices of $S$ are connected by an arc in $D$,
and (ii) every vertex outside $S$ can reach $S$ by an arc in $D$.
This notion was originally introduced by von Neumann and Morgenstern \cite{VM1944} in 1944.
Since it has many applications in both cooperative games and logic, see \cite{Berge1977,Berge1984},
its existence has been the focus of extensive study,
both from the algorithmic perspective and the sufficient condition perspective,
see \cite{BD1990,BG2006,Chvatal1973,CL1974,Duchet1980,GN1984,Richardson1953,VM1944} for more results.
Afterwards, Sands et al. \cite{SSW1982} proposed the notion of kernels by monochromatic paths in 1982.
For an arc-colored digraph $D$,
define its {\em kernel by monochromatic paths} (or {\em MP-kernel} for short) to be a set $S$ of vertices such that
(i) no two vertices of $S$ are connected by a monochromatic path in $D$,
and (ii) every vertex outside $S$ can reach $S$ by a monochromatic path in $D$.
It is worth noting that fruitful results on MP-kernels have been obtained in the past decades,
we refer the reader to \cite{DM1983,Galeana1996,GR2004,GR2004285,SSW1982,Shen1988}.

Based on the results and problems on kernels and MP-kernels,
Bai et al. \cite{BFZ2017} and independently Delgado-Escalante et al. \cite{DGO2018} considered the existence of PCP-kernels in arc-colored digraphs in 2018.
Define a {\em kernel by properly colored paths} (or {\em PCP-kernel} for short) of an arc-colored digraph $D$ to be a set $S$ of vertices such that
(i) no two vertices of $S$ are connected by a properly colored path in $D$,
and (ii) every vertex outside $S$ can reach $S$ by a properly colored path in $D$.
A {\em tournament} is an orientation of a complete graph
and a {\em bipartite tournament} ({\em multipartite tournament}) is an orientation of a complete bipartite (multipartite) graph.
A digraph $D$ is {\em semi-complete} if for every two vertices there exists at least one arc between them,
and $D$ is {\em quasi-transitive} if for every two vertices $u,w$ with $uv,vw\in A(D)$ there exists at least one arc between them.
Bai et al. \cite{BFZ2017} conjectured that
every arc-colored digraph with all cycles properly colored has a PCP-kernel
and verified the conjecture for digraphs with no intersecting cycles, semi-complete digraphs and bipartite tournaments, respectively;
moreover, weaker conditions for the latter two classes of digraphs are given
and it is shown to be NP-hard to decide whether an arc-colored digraph has a PCP-kernel.
In \cite{DGO2018} Delgado-Escalante et al. provided some sufficient conditions for the existence of PCP-kernels in tournaments, quasi-transitive digraphs and multipartite tournaments, respectively.

In the concluding section of \cite{BFZ2017}, Bai et al. introduced the notion of rainbow kernels but did not present detailed analysis.
Here a {\em kernel by rainbow paths} (or {\em rainbow kernel} for short)
of an arc-colored digraph $D$ is defined,
similar to the definition of MP-kernels or PCP-kernels,
to be a set $S$ of vertices such that
(i) no two vertices of $S$ are connected by a rainbow path in $D$,
and (ii) every vertex outside $S$ can reach $S$ by a rainbow path in $D$.
In this paper we concentrate on rainbow kernels in arc-colored tournaments.
Here note that the existence of a rainbow kernel in an arc-colored tournament is equivalent to the existence of one vertex such that
all other vertices can reach this vertex by a rainbow path.
Before we present our main results,
we give some observations on rainbow kernels in sake of a better understanding of this notion.

\begin{observation}\label{observation: rainbow kernel in acyclic digraphs}
Every arc-colored acyclic digraph has a rainbow kernel.
\end{observation}

This is due to the fact that every subdigraph of an acyclic digraph has at least one sink
and we can obtain a rainbow kernel $S$ by the following progress.
First put the sinks of the digraph into $S$ and delete all the vertices that can reach $S$ by a rainbow path,
then put the sinks of the remaining subdigraph into $S$,
continue the above process till there is no vertex left.
The resulting set $S$ is a required rainbow kernel.

Note that the notion of rainbow kernels in arc-colored digraphs generalizes, in some sense, the concept of kernels in digraphs.
On one hand,
a digraph $D$ has a kernel if and only if a monochromatic $D$ has a rainbow kernel;
on the other hand,
an arc-colored digraph $D$ has a rainbow kernel if and only if its rainbow closure $\mathscr{C}_{_{R}}(D)$ has a kernel,
where the {\em rainbow closure} $\mathscr{C}_{_{R}}(D)$ of $D$ is defined
to be a digraph with vertex set $V(\mathscr{C}_{_{R}}(D))=V(D)$
and arc set
$$
A(\mathscr{C}_{_{R}}(D))=\{uv:~there~is~a~rainbow~(u,v)\textrm{-}path~in~D\}.
$$

In the rest of this paper,
we first show the NP-completeness of the rainbow kernel problem in arc-colored tournaments
and then provide a sufficient condition for the existence of a rainbow kernel in this digraph class.

\section{Main results}

Chv\'{a}tal \cite{Chvatal1973} showed in 1973 that it is NP-complete to decide whether a digraph has a kernel.
Bai et al. \cite{BFZ2017} pointed out in 2018 that it is NP-hard to recognize whether an arc-colored digraph has a rainbow kernel.
In this paper we study the complexity of the rainbow kernel problem in arc-colored tournaments.
In general, things look better in dense instances.
For example,
we shall show that the PCP-kernel problem in tournaments is polynomially solvable,
but in general digraphs it has been proved to be NP-hard in \cite{BFZ2017}.
However, for the rainbow kernel problem,
it will be shown to be NP-complete even when we restrict the digraph class to arc-colored tournaments.

\begin{theorem}\label{thm: RKT is NPC}
It is NP-complete to decide whether an arc-colored tournament has a rainbow kernel.
\end{theorem}

Call a digraph {\em strongly connected} if there exists an $(x,y)$-path for every two distinct vertices $x$ and $y$ of the digraph.
If a rainbow tournament is strongly connected,
then each vertex forms a rainbow kernel;
if it is not strongly connected,
then each vertex in its `sink component' forms a rainbow kernel.
It therefore follows that every rainbow tournament has a rainbow kernel.

Note that an arc-colored tournament $T$ is rainbow if and only if
every $k$-vertex subtournament, $2\leq k\leq |V(T)|$, of $T$ is colored with $\binom{k}{2}=\frac{k^{2}-k}{2}$ colors.
A natural question is what will happen if we reduce the number of colors.
We notice that if every triangle is rainbow then the tournament has a rainbow kernel (see the sketch of the proof in Section \ref{section: concluding}).
As an immediate consequence,
we get that if every strongly connected $k$-vertex subtournament is colored with at least $k$ colors,
$3\leq k\leq |V(T)|$,
then the tournament $T$ has a rainbow kernel.
How about that we continue to reduce the number of colors?
For example, we only require that every $k$-vertex subtournament is colored with at least $k-1$ colors.
In this paper we show that it still suffices for the existence of a rainbow kernel.

\begin{theorem}\label{thm: rainbow kernels in tournaments}
Let $T$ be an arc-colored tournament.
If every strongly connected $k$-vertex subtournament of $T$ is colored with at least $k-1$ colors,
$3\leq k\leq |V(T)|$,
then $T$ has a rainbow kernel.
\end{theorem}

The number $k-1$ in Theorem \ref{thm: rainbow kernels in tournaments} cannot be reduced in view of the arc-colored tournament $T_{n}^{*}$ defined as follows.
Let $V(T_{n}^{*})=\{v_{0},v_{1},\ldots,v_{n-1}\}$
and let $v_{i}\rightarrow v_{j}$ for any $0\leq i<j\leq n-1$, except that $v_{2}\rightarrow v_{0}$.
Color the arcs $v_{0}v_{1},v_{1}v_{2},v_{2}v_{0}$ with the same color and for other arcs color them arbitrarily.
It is not difficult to see that every strongly connected $k$-vertex subtournament, $3\leq k\leq n$, of $T_{n}^{*}$ is colored with $k-2$ colors but $T_{n}^{*}$ has no rainbow kernel.

Note that the monochromatic triangle is an `unfriendly' structure for the existence of a rainbow kernel in $T_{n}^{*}$.
One may ask whether `no monochromatic triangles' suffices in general.
The answer is negative as can be seen from the tournament $T_{5}^{*}$ in Figure \ref{figure: a 2-colored T5},
in which solid arcs and dotted arcs represent arcs colored with two distinct colors.
But it seems that if we forbid monochromatic triangles
then the number of colors required in Theorem \ref{thm: rainbow kernels in tournaments} may be reduced.
Let $f(k)$ be the minimum integer such that
every arc-colored tournament with no monochromatic triangles and
with all its strongly connected $k$-vertex subtournaments, $4\leq k\leq |V(T)|$, colored with at least $f(k)$ colors has a rainbow kernel.
We conjecture that $f(k)\leq k-2$.

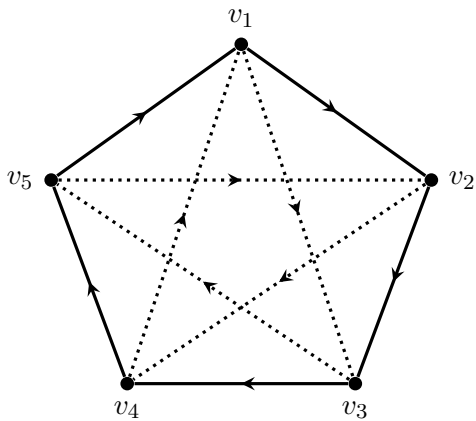
\begin{figure}[ht]
\begin{center}
\begin{tikzpicture}
\tikzstyle{vertex}=[circle,inner sep=1.8pt, minimum size=0.1pt]


\vertex (a)[fill] at (0,2)[label=above:$v_{1}$]{};
\vertex (b)[fill] at (2.5,0.2)[label=right:$v_{2}$]{};
\vertex (c)[fill] at (1.5,-2.5)[label=below:$v_{3}$]{};
\vertex (d)[fill] at (-1.5,-2.5)[label=below:$v_{4}$]{};
\vertex (e)[fill] at (-2.5,0.2)[label=left:$v_{5}$]{};


\draw [directed,line width=1.2pt] (a)--(b);
\draw [directed,line width=1.2pt] (b)--(c);
\draw [directed,line width=1.2pt] (c)--(d);
\draw [directed,line width=1.2pt] (d)--(e);
\draw [directed,line width=1.2pt] (e)--(a);
\draw [directed,dotted,line width=1.2pt] (a)--(c);
\draw [directed,dotted,line width=1.2pt] (c)--(e);
\draw [directed,dotted,line width=1.2pt] (e)--(b);
\draw [directed,dotted,line width=1.2pt] (b)--(d);
\draw [directed,dotted,line width=1.2pt] (d)--(a);

\end{tikzpicture}
\end{center}
\caption{A 2-colored tournament $T_{5}^{*}$ with no rainbow kernel.}
\label{figure: a 2-colored T5}
\end{figure}

In \cite{BFZ2017} Bai et al. provided a sufficient condition for the existence of a PCP-kernel in arc-colored tournaments.
Another main contribution of this paper concentrates on the complexity of the above problem.

\begin{theorem}\label{thm: PKT is polynomial}
It is polynomially solvable for finding a PCP-kernel in an arc-colored tournament or deciding that no PCP-kernel exists.
\end{theorem}

We present the proofs of Theorems \ref{thm: RKT is NPC}, \ref{thm: rainbow kernels in tournaments} and \ref{thm: PKT is polynomial}
in Sections \ref{section: proof of thm 1}, \ref{section: proof of thm 2} and \ref{section: proof of thm 3}, respectively.
Here we give some necessary definitions and notations.
For an arc $uv$ of a digraph $D$,
we sometimes write $u\rightarrow v$ and say $u$ {\em dominates} $v$.
For two vertex-disjoint subsets $M$ and $N$ of $V(D)$,
we write $M\rightarrow N$ if each vertex of $M$ dominates all vertex of $N$.
For two subsets $R$ and $S$ of $V(D)$,
we use $N^{+}_{S}(R)$ (resp. $N^{-}_{S}(R)$) to denote the set of outneighbors (resp. inneighbors) of the vertices of $R$ in $S$.
For convenience, we write $v\rightarrow M$ for $\{v\}\rightarrow M$,
$M\rightarrow v$ for $M\rightarrow \{v\}$,
$N^{+}_{S}(v)$ for $N^{+}_{S}(\{v\})$
and $N^{-}_{S}(v)$ for $N^{-}_{S}(\{v\})$.

\section{Proof of Theorem \ref{thm: RKT is NPC}}
\label{section: proof of thm 1}

We first offer some related definitions.
Define a {\em hypergraph} $H$ to be an ordered pair $H=(V,E)$,
where $V$ is a finite set (the {\em vertex set})
and $E$ is a family of distinct subsets of $V$ (the {\em edge set}).
A hypergraph is {\em $k$-uniform} if all edges have size $k$,
where $k\geq 2$ is an integer.
A {\em matching} of a hypergraph is a set of disjoint edges, i.e., no two of them have a common vertex.
Call a matching of a hypergraph {\em perfect} if it contains all the vertices.
Note that a perfect matching of a $k$-uniform hypergraph has size $|V(H)|/k$
and a necessary condition for its existence is that $|V(H)|=0$ (mod $k$).\\

\noindent
\textbf{$k$-Dimensional Perfect Matching ($k$DPM)}\\
\textbf{Input:} A $k$-uniform hypergraph $H$ on $n$ vertices.\\
\textbf{Question:} Does $H$ have a $k$-dimensional perfect matching?\\

The $k$-dimensional perfect matching problem is known to be NP-complete for every $k\geq 3$, see \cite{Karp1972,Papadimitriou1994}.
Here we only use the result that $3DPM\in NPC$.
Define an {\em oriented graph} to be a simple digraph without 2-cycles.
The following decision problem will play an important role in our proof.\\

\noindent
\textbf{Rainbow Paths in Oriented Graphs (RPOG)}\\
\textbf{Input:} An arc-colored oriented graph $D$ with two distinct vertices $x$ and $y$.\\
\textbf{Question:} Does $D$ have a rainbow path from $x$ to $y$?\\

For convenience,
we state the rainbow kernel problem as follows.\\

\noindent
\textbf{Rainbow Kernels in Tournaments (RKT)}\\
\textbf{Input:} An arc-colored tournament $T$.\\
\textbf{Question:} Does $T$ have a rainbow kernel?\\

It is not difficult to see that $RKT\in NP$.
The theorem will be shown by a reduction from the RPOG problem,
whose NP-hardness will be obtained by a reduction from the 3DPM problem, which is known to be NP-complete.

We first show the NP-hardness of the RPOG problem.
For any instance $H$ of the 3DPM problem,
we will construct an arc-colored oriented graph $D_{_H}$ with two vertices $x$ and $y$, an instance of the RPOG problem,
in such a way that $H$ has a 3-dimensional perfect matching if and only if
$D_{_{H}}$ has a rainbow path from $x$ to $y$.

Let $H$ be an arbitrary $3n$-vertex 3-uniform hypergraph
and denote its vertex set and edge set by $\{1,\ldots,3n\}$ and $\{h_{1},\ldots,h_{m}\}$, respectively.
We construct an arc-colored oriented graph $D_{_H}$ with $\{x_{0},x_{1},\ldots,x_{n}\}\subseteq V(D_{_H})$
and add $m$ internally-vertex-disjoint paths $P^{1}_{i+1},\ldots,P^{m}_{i+1}$, each of length 3, from $x_{i}$ to $x_{i+1}$,
for each $0\leq i\leq n-1$, see Figure \ref{figure: oriented graph}.
Note that $D_{_H}$ has $2nm+n+1$ vertices and $3nm$ arcs.
We color the arcs of the path $P_{i}^{j}$ with vertices contained in the edge $h_{j}$ for each $1\leq i\leq n$ and $1\leq j\leq m$.
Setting $x=x_{0}$ and $y=x_{n}$.
We now show that $H$ has a 3-dimensional perfect matching if and only if
$D_{_{H}}$ has a rainbow path from $x$ to $y$.
If $D_{_{H}}$ has a rainbow $(x,y)$-path,
say $P=\cup_{i=1}^{n}P_{i}^{j_{i}}$,
then for any $i'\neq i''$ we have $h_{j_{i'}}\cap h_{j_{i''}}=\emptyset$,
since otherwise two arcs with the same color will appear.
Clearly, the edges (of $H$) corresponding to the paths $P_{i}^{j_{i}}$, $1\leq i\leq n$, are disjoint and they cover all vertices of $H$.
The reverse can be obtained similarly.

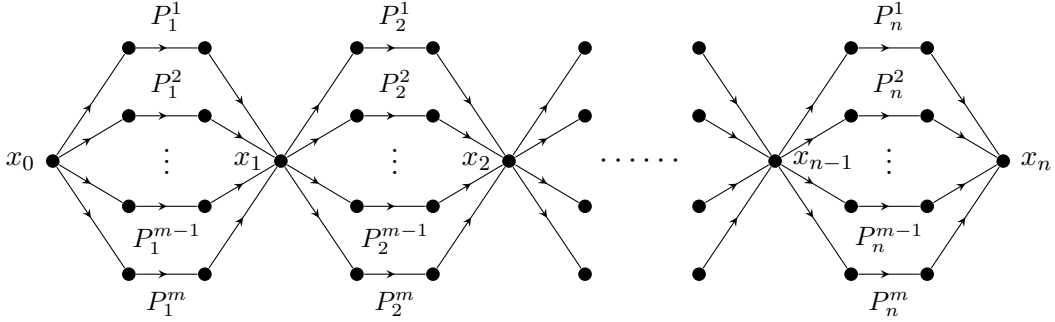
\begin{figure}[ht]
\begin{center}
\begin{tikzpicture}
\tikzstyle{vertex}=[circle,inner sep=1.8pt, minimum size=0.1pt]


\vertex (x2) [fill] at (0,0)[label=left:$x_{2}$]{};
\vertex (x1) [fill] at (-3,0)[label=left:$x_{1}$]{};
\vertex (x0) [fill] at (-6,0)[label=left:$x_{0}$]{};
\vertex (x3) [fill] at (3.5,0)[label=right:$x_{n-1}$]{};
\vertex (x4) [fill] at (6.5,0)[label=right:$x_{n}$]{};


\vertex (y11) [fill] at (-2,1.5)[]{};
\vertex (y12) [fill] at (-2,0.6)[]{};
\vertex (y13) [fill] at (-2,-0.6)[]{};
\vertex (y14) [fill] at (-2,-1.5)[]{};

\vertex (y15) [fill] at (-1,1.5)[]{};
\vertex (y16) [fill] at (-1,0.6)[]{};
\vertex (y17) [fill] at (-1,-0.6)[]{};
\vertex (y18) [fill] at (-1,-1.5)[]{};

\path (x1) edge[directed] (y11);
\path (x1) edge[directed] (y12);
\path (x1) edge[directed] (y13);
\path (x1) edge[directed] (y14);

\path (y11) edge[directed] (y15);
\path (y12) edge[directed] (y16);
\path (y13) edge[directed] (y17);
\path (y14) edge[directed] (y18);

\path (y15) edge[directed] (x2);
\path (y16) edge[directed] (x2);
\path (y17) edge[directed] (x2);
\path (y18) edge[directed] (x2);


\vertex (y01) [fill] at (-5,1.5)[]{};
\vertex (y02) [fill] at (-5,0.6)[]{};
\vertex (y03) [fill] at (-5,-0.6)[]{};
\vertex (y04) [fill] at (-5,-1.5)[]{};

\vertex (y05) [fill] at (-4,1.5)[]{};
\vertex (y06) [fill] at (-4,0.6)[]{};
\vertex (y07) [fill] at (-4,-0.6)[]{};
\vertex (y08) [fill] at (-4,-1.5)[]{};

\path (x0) edge[directed] (y01);
\path (x0) edge[directed] (y02);
\path (x0) edge[directed] (y03);
\path (x0) edge[directed] (y04);

\path (y01) edge[directed] (y05);
\path (y02) edge[directed] (y06);
\path (y03) edge[directed] (y07);
\path (y04) edge[directed] (y08);

\path (y05) edge[directed] (x1);
\path (y06) edge[directed] (x1);
\path (y07) edge[directed] (x1);
\path (y08) edge[directed] (x1);


\vertex (y31) [fill] at (4.5,1.5)[]{};
\vertex (y32) [fill] at (4.5,0.6)[]{};
\vertex (y33) [fill] at (4.5,-0.6)[]{};
\vertex (y34) [fill] at (4.5,-1.5)[]{};

\vertex (y35) [fill] at (5.5,1.5)[]{};
\vertex (y36) [fill] at (5.5,0.6)[]{};
\vertex (y37) [fill] at (5.5,-0.6)[]{};
\vertex (y38) [fill] at (5.5,-1.5)[]{};

\path (x3) edge[directed] (y31);
\path (x3) edge[directed] (y32);
\path (x3) edge[directed] (y33);
\path (x3) edge[directed] (y34);

\path (y31) edge[directed] (y35);
\path (y32) edge[directed] (y36);
\path (y33) edge[directed] (y37);
\path (y34) edge[directed] (y38);

\path (y35) edge[directed] (x4);
\path (y36) edge[directed] (x4);
\path (y37) edge[directed] (x4);
\path (y38) edge[directed] (x4);


\vertex (y21) [fill] at (1,1.5)[]{};
\vertex (y22) [fill] at (1,0.6)[]{};
\vertex (y23) [fill] at (1,-0.6)[]{};
\vertex (y24) [fill] at (1,-1.5)[]{};

\vertex (y25) [fill] at (2.5,1.5)[]{};
\vertex (y26) [fill] at (2.5,0.6)[]{};
\vertex (y27) [fill] at (2.5,-0.6)[]{};
\vertex (y28) [fill] at (2.5,-1.5)[]{};

\path (x2) edge[directed] (y21);
\path (x2) edge[directed] (y22);
\path (x2) edge[directed] (y23);
\path (x2) edge[directed] (y24);

\path (y25) edge[directed] (x3);
\path (y26) edge[directed] (x3);
\path (y27) edge[directed] (x3);
\path (y28) edge[directed] (x3);


\vertex at (1.75,0)[label=center:{\large $\cdots \cdots$}]{};
\vertex at (-1.5,0.1)[label=center:{\large $\vdots$}]{};
\vertex at (-4.5,0.1)[label=center:{\large $\vdots$}]{};
\vertex at (5,0.1)[label=center:{\large $\vdots$}]{};


\vertex at (-4.5,1.5)[label=above:{{\small $P_{1}^{1}$}}]{};
\vertex at (-4.5,0.6)[label=above:{{\small $P_{1}^{2}$}}]{};
\vertex at (-4.5,-0.6)[label=below:{{\small $P_{1}^{m-1}$}}]{};
\vertex at (-4.5,-1.5)[label=below:{{\small $P_{1}^{m}$}}]{};

\vertex at (-1.5,1.5)[label=above:{{\small $P_{2}^{1}$}}]{};
\vertex at (-1.5,0.6)[label=above:{{\small $P_{2}^{2}$}}]{};
\vertex at (-1.5,-0.6)[label=below:{{\small $P_{2}^{m-1}$}}]{};
\vertex at (-1.5,-1.5)[label=below:{{\small $P_{2}^{m}$}}]{};

\vertex at (5,1.5)[label=above:{{\small $P_{n}^{1}$}}]{};
\vertex at (5,0.6)[label=above:{{\small $P_{n}^{2}$}}]{};
\vertex at (5,-0.6)[label=below:{{\small $P_{n}^{m-1}$}}]{};
\vertex at (5,-1.5)[label=below:{{\small $P_{n}^{m}$}}]{};


\end{tikzpicture}
\end{center}
\caption{The arc-colored oriented graph $D_{_H}$.}
\label{figure: oriented graph}
\end{figure}

With the NP-hardness of the RPOG problem on hand, we are now ready to complete the NP-hardness of the RKT problem.
For any instance $D$ with two distinct vertices $x$ and $y$ of the RPOG problem,
we will construct a tournament $T_{_D}$ in such a way that $T_{_{D}}$ has a rainbow kernel if and only if $D$ has a rainbow $(x,y)$-path.

Let $D$ be an arbitrary arc-colored oriented graph with two distinct vertices $x$ and $y$.
Let $\alpha, \beta,\gamma,\omega$ be four colors not appeared in $D$.
We construct an arc-colored tournament $T_{_D}$ with
$$V(T_{_D})=V(D)\cup \{x',x'',y',y''\}.$$
Let $N^{+}(x')=\{x,y',y''\}$, $N^{-}(x')=\{x''\}\cup V(D)\backslash \{x\}$,
$N^{+}(x'')=\{x'\}$, $N^{-}(x'')=V(D)\cup \{y',y''\}$,
$N^{+}(y')=\{x'',y''\}$, $N^{-}(y')=V(D)\cup \{x'\}$,
$N^{+}(y'')=\{x''\}$ and $N^{-}(y'')=V(D)\cup \{x',y'\}$.
For any two non-adjacent vertices in $D$ we add an arbitrary arc and color it with color $\alpha$.
Color the arcs from $V(D)$ to $\{x'',y',y''\}$ and the arc $x'x$ with color $\alpha$.
Color the arcs from $V(D)$ to $x'$ and the arcs in $\{x''x',x'y',x'y'',y'x'',y''x''\}$ with color $\beta$.
Color the arc $yy'$ with color $\gamma$ and the arc $y'y''$ with color $\omega$.
The resulting digraph is an arc-colored tournament,
see Figure \ref{figure: tournament},
in which dashed arcs, dotted arcs, thick dotted arcs and solid arcs represent arcs colored with colors $\alpha, \beta, \gamma, \omega$, respectively.
It therefore suffices to show the following.

\begin{figure}[ht]
\begin{center}
\begin{tikzpicture}
\tikzstyle{vertex}=[circle,inner sep=2pt, minimum size=0.1pt]


\draw (0,0) [line width=1.5pt] ellipse (2.6 and 0.9);


\vertex (x) [fill] at (-2.6,0)[]{};
\vertex at (-2.7,-0.1)[label=below:$x$]{};
\vertex (x1) [fill] at (-4.1,0)[label=below:$x'$]{};
\vertex (x2) [fill] at (-5.6,0)[]{};
\vertex at (-5.8,0)[label=below:$x''$]{};

\vertex (y) [fill] at (2.6,0)[]{};
\vertex at (2.7,-0.1)[label=below:$y$]{};
\vertex (y1) [fill] at (4.1,0)[]{};
\vertex at (4.2,0)[label=below:$y'$]{};
\vertex (y2) [fill] at (5.6,0)[]{};
\vertex at (5.7,0)[label=below:$y''$]{};

\vertex (q1) at (5.5,3)[]{};
\vertex (q2) at (7.05,3)[]{};
\path (q1) edge[dashed,line width=1.2pt] (q2);
\vertex (i) at (7.1,3)[label=right:$\alpha$]{};

\vertex (q3) at (5.5,2.6)[]{};
\vertex (q4) at (7.1,2.6)[]{};
\path (q3) edge[dotted,line width=1.1pt] (q4);
\vertex (i) at (7.1,2.6)[label=right:$\beta$]{};

\vertex (q5) at (5.5,2.2)[]{};
\vertex (q6) at (7.1,2.2)[]{};
\path (q5) edge[dotted,line width=2pt] (q6);
\vertex (i) at (7.1,2.2)[label=right:$\gamma$]{};

\vertex (q7) at (5.5,1.8)[]{};
\vertex (q8) at (7.1,1.8)[]{};
\path (q7) edge[line width=1.2pt] (q8);
\vertex (i) at (7.1,1.8)[label=right:$\omega$]{};

\path (x2) edge[directed,dotted,line width=1.1pt] (x1);
\path (x1) edge[directed,dashed,line width=1.1pt] (x);

\path (y) edge[directed,dotted,line width=2pt] (y1);

\path (y1) edge[directed,line width=1pt] (y2);

\path (y2) edge[directed,dotted, bend left=60,line width=1.1pt] (x2);
\path (y1) edge[directed,dotted, bend left=50,line width=1.1pt] (x2);

\path (x1) edge[directed,dotted, bend left=80,line width=1.1pt] (y1);
\path (x1) edge[directed,dotted, bend left=80,line width=1.1pt] (y2);


\vertex (z) [fill] at (0,0.9)[]{};
\vertex at (0,0)[label=center:$D$]{};

\path (z) edge[directed,dashed, bend right=80,line width=1.2pt] (x2);
\path (z) edge[directed,dashed, bend left=50,line width=1.2pt] (y1);
\path (z) edge[directed,dashed, bend left=70,line width=1.2pt] (y2);
\path (z) edge[directed,dotted, bend right=50,line width=1.1pt] (x1);


\vertex (u1) [fill] at (-1.8,0.2)[]{};
\vertex (u2) [fill] at (-0.8,0.2)[]{};

\vertex (u3) [fill] at (-1.8,-0.2)[]{};
\vertex (u4) [fill] at (-0.8,-0.2)[]{};

\vertex (v1) [fill] at (0.8,0.2)[]{};
\vertex (v2) [fill] at (1.8,0.2)[]{};

\vertex (v3) [fill] at (0.8,-0.2)[]{};
\vertex (v4) [fill] at (1.8,-0.2)[]{};

\path (u1) edge[directed,dashed,line width=1.2pt] (u2);
\path (v1) edge[directed,dashed,line width=1.2pt] (v2);

\path (u4) edge[directed,dashed,line width=1.2pt] (u3);
\path (v4) edge[directed,dashed,line width=1.2pt] (v3);


\end{tikzpicture}
\end{center}
\caption{The arc-colored tournament $T_{_D}$.}
\label{figure: tournament}
\end{figure}
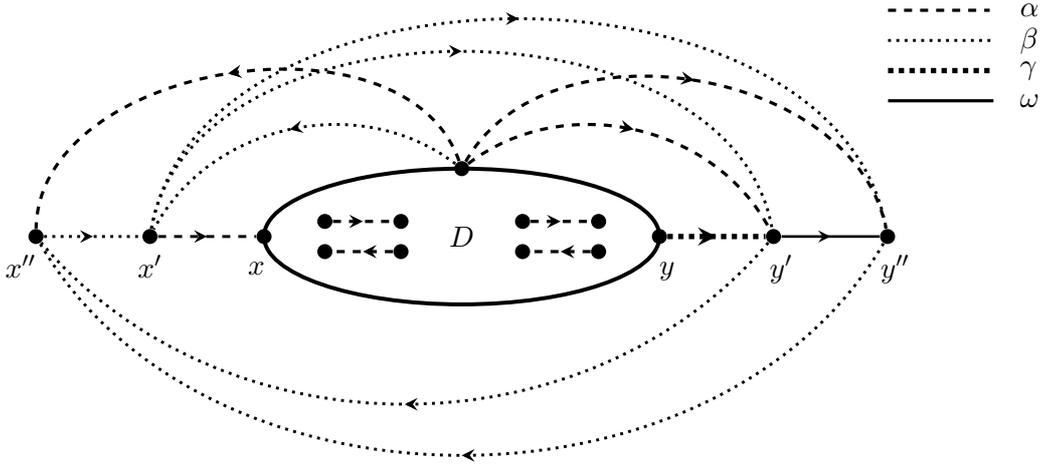

\begin{claim}
The tournament $T_{_D}$ has a rainbow kernel if and only if $D$ has a rainbow $(x,y)$-path.
\end{claim}

\begin{proof}
We first show the sufficiency.
Note that if a tournament has a rainbow kernel then it contains exactly one vertex.
Assume that a rainbow $(x,y)$-path $P^{*}$ in $D$ exists.
We deduce that $y''$ forms a rainbow kernel.
One can see that $x''x'xP^{*}yy'y''$ is a rainbow $(x'',y'')$-path.
So it suffices to consider the vertices not in the above rainbow $(x'',y'')$-path,
which can reach $y''$ by an arc and is clearly a rainbow path, see Figure \ref{figure: tournament}.

Now we show the necessarity.
Assume that $T_{_{D}}$ has a rainbow kernel.
Since $N^{+}(y'')=\{x''\}$, $N^{+}(x'')=\{x'\}$ and $c(y''x'')=c(x''x')=\beta$,
we can get that $y''$ cannot reach any vertex in $V(T_{_D})\backslash \{x'',y''\}$ by a rainbow path.
This implies that the possible rainbow kernel consists of either $x''$ or $y''$.
We complete the proof by showing that a rainbow $(x,y)$-path in $D$ will appear in each of the above two cases.

If $x''$ forms a rainbow kernel of $T_{_D}$,
then, by definition, there exists a rainbow $(x',x'')$-path, say $P'=x'u_{1}\cdots u_{s}x''$.
If $c(u_{s}x'')=\alpha$,
then, since $N^{+}(x')=\{x,y',y''\}$ and $c(x'x)=\alpha$,
we have $u_{1}\in \{y',y''\}$.
If $u_{1}=y'$,
then, since $N^{+}(y')=\{y''\}$ and $N^{+}(y'')=\{x''\}$,
the path will be $x'\rightarrow y'\rightarrow y''\rightarrow x''$,
which is clearly not a rainbow path, a contradiction.
If $u_{1}=y''$,
then, since $N^{+}(y'')=\{x''\}$,
the path will be $x'\rightarrow y''\rightarrow x''$,
which is clearly not a rainbow path either, a contradiction.
Now assume that $c(u_{s}x'')=\beta$.
It follows that $u_{s}\in \{y',y''\}$ and $u_{1}=x$.
Note that the colors $\alpha$ and $\beta$ have appeared on $P'$ and the path $P'$ is rainbow.
Thus, if $u_{s}=y'$ then $u_{s-1}=y$;
if $u_{s}=y''$ then $u_{s-1}=y'$, $u_{s-2}=y$.
Moreover, in either of the above two cases, there exists a rainbow $(x,y)$-path with no arcs colored with $\alpha$.
This implies that $D$ has a rainbow $(x,y)$-path.

If $y''$ forms a rainbow kernel of $T_{_D}$,
then $T_{_D}$ has a rainbow $(x'',y'')$-path, say $P''=x''\rightarrow v_{1}\rightarrow \cdots \rightarrow v_{t}\rightarrow y''$.
Since $N^{+}(x'')=\{x'\}$,
we have $v_{1}=x'$.
Similarly, since $c(x''x')=\beta$, $N^{+}(x')=\{x,y',y''\}$ and $c(x'y')=c(x'y'')=\beta$,
we have $v_{2}=x$.
Now note that both $\alpha$ and $\beta$ have appeared on $P''$.
So we have $v_{k-1}=y'$ and $v_{k-2}=y$.
It therefore follows that there exists a rainbow $(x,y)$-path in $T_{_D}$ with no arcs colored with $\alpha$,
which implies that $D$ has a rainbow $(x,y)$-path.
\end{proof}

We are left to show that the two reductions are polynomial,
which follows directly from our constructions.

\section{Proof of Theorem \ref{thm: rainbow kernels in tournaments}}
\label{section: proof of thm 2}

For convenience,
we shall call a vertex $v$ {\em good}
if all other vertices can reach $v$ by a rainbow path.
Let $T$ be an arc-colored $n$-vertex tournament.
We prove the result by induction on $n$.

The result clearly holds for $n=3$.
Now assume that $n\geq 4$ and the result holds for every arc-colored tournament with order at most $n-1$.
For each vertex $v$ of $T$,
by induction hypothesis,
the subtournament $T-v$ has a good vertex.
Denote by $v^*$ the good vertex of $T-v$ corresponding to the given coloring of $T$.
Then $v^{*}\rightarrow v$,
since otherwise $v^{*}$ is a good vertex of $T$.
For two distinct vertices $u$ and $v$,
we claim that $u^*\neq v^*$.
If not,
then by the definition of $u^{*}$
there exist a rainbow $(v,u^{*})$-path in $T-u$ and a rainbow $(u,u^{*})$-path in $T-v$.
It follows immediately that there exist a rainbow $(v,u^{*})$-path and a rainbow $(u,u^{*})$-path in $T$.
Thus, $u^{*}$ is a good vertex of $T$, a contradiction.

Now consider the subdigraph $H$ induced by the arc set $\{v^{*}v:v\in V(T)\}$.
Since each vertex of $H$ has indegree at least one and outdegree at most one,
then $H$ consists of vertex-disjoint cycles.
If $H$ has at least two cycles,
then by induction hypothesis the induced subtournament on each cycle has a good vertex,
which is obviously a good vertex of $T$, a contradiction.
So $H$ consists of one cycle, i.e., $H$ is a Hamilton cycle of $T$.

Let $H=(v_{0},v_{1},\ldots,v_{n-1},v_{0})$.
By the choices of the arcs,
there exists no rainbow $(v_{i},v_{i-1})$-path in $T$ for any $0\leq i\leq n-1$,
where the subtraction is modulo $n$.
Consider the three vertices $v_{0},v_{1},v_{2}$,
if $v_{2}\rightarrow v_{0}$,
then, since each triangle is colored with at least two colors,
we have either $v_{2}\rightarrow v_{0}\rightarrow v_{1}$ is a rainbow $(v_{2},v_{1})$-path
or $v_{1}\rightarrow v_{2}\rightarrow v_{0}$ is a rainbow $(v_{1},v_{0})$-path, a contradiction.
So $v_{0}\rightarrow v_{2}$.
In fact,
one can see from the above simple proof that $v_{i}\rightarrow v_{i+2}$ for any $0\leq i\leq n-1$,
where the addition is modulo $n$.

Let $t$ be the maximum integer such that
$v_{i}\rightarrow \{v_{i+1},v_{i+2},\ldots,v_{i+t-1}\}$ for any $0\leq i\leq n-1$, the addition is modulo $n$.
By the maximality of $t$,
there exists an integer $j$ such that
$v_{j}\rightarrow \{v_{j+1},v_{j+2},\ldots,v_{j+t-1}\}$ and $v_{j+t}\rightarrow v_{j}$.
Assume w.l.o.g. that $j=0$.
We will complete the proof by getting a contradiction that there exists a rainbow $(v_{i},v_{i-1})$-path for some $i\in \{1,\ldots,t\}$,
which follows directly from the lemma below.

\begin{lemma}\label{lemma: T*}
Let $T^{*}$ be an arc-colored tournament on vertex set $\{v_{0},v_{1},\ldots,v_{t}\}$ satisfying that
$v_{i}\rightarrow v_{j}$ for any $0\leq i<j\leq t$ with only one exception that $v_{t}\rightarrow v_{0}$,
and every strongly connected $k$-vertex subtournament $T_{k}^{*}$ of $T^{*}$, $3\leq k\leq t$, is colored with at least $k-1$ colors.
Then $T^{*}$ has a rainbow $(v_{i},v_{i-1})$-path for some $i\in \{1,\ldots,t\}$.
\end{lemma}

\begin{proof}
Assume that $T^{*}$ is a counterexample with minimum number of vertices.
By the definition of $T^{*}$,
every strongly connected subtournament of $T^{*}$ contains the arc $v_{t}v_{0}$.
This observation will be used several times in this proof.
Now we first present the following result.

\begin{itemize}
  \item There exists a rainbow $(v_{i},v_{j})$-path for any $2\leq i\leq t$ and $0\leq j\leq i-2$.
\end{itemize}

Assume the opposite that there exists no rainbow $(v_{i},v_{j})$-path for some $i,j$ with $i-j\geq 2$.
Then $M=\{v_{j+1},\ldots,v_{i-1}\}\neq \emptyset$.
Let $T'$ be the subtournament induced by $V(T^{*})\backslash M$.
Clearly, $|V(T')|<|V(T^{*})|$,
the tournament $T'$ satisfies the coloring conditions in the claim and is also a counterexample of the claim, a contradiction to the minimality of $T^{*}$.
Besides, we can get the following useful result.

\begin{itemize}
  \item $c(v_{i}v_{t})\neq c(v_{t-1}v_{t})$ for each $1\leq i\leq t-2$.
\end{itemize}

Assume the opposite that $c(v_{i}v_{t})=c(v_{t-1}v_{t})$ for some $1\leq i\leq t-2$.
Since $|(t-1)-(i-1)|=|t-i|\geq 2$,
by the above analysis,
$T^{*}$ has a rainbow $(v_{t-1},v_{i-1})$-path $P$.
It is not difficult to see that $P$ must pass through the arc $v_{t-1}v_{t}$.
So the path $P'$ obtained from $P$ by replacing $v_{t-1}v_{t}$ with $v_{i}v_{t}$ is obviously a rainbow $(v_{i},v_{i-1})$-path in $T^{*}$, a contradiction.

Now assume w.l.o.g. that $c(v_{t}v_{0})=\alpha$.
Since $v_{t}\rightarrow v_{0}$, $v_{0}\rightarrow v_{t-1}$ and $T^{*}$ has no rainbow $(v_{t},v_{t-1})$-path,
we have $c(v_{0}v_{t-1})=c(v_{t}v_{0})=\alpha$.
Note that $(v_{0},v_{t-1},v_{t},v_{0})$ is a triangle
and each triangle is colored with at least 2 colors.
So $c(v_{t-1}v_{t})\neq \alpha$ and we can let $c(v_{t-1}v_{t})=\beta$.
It then follows that $c(v_{i}v_{t})\neq \beta$ for each $1\leq i\leq t-2$.

\begin{itemize}
  \item There exists some $2\leq i\leq t-2$ such that $c(v_{i}v_{t-1}),c(v_{i}v_{t}),\alpha,\beta$ are four distinct colors.
\end{itemize}

If not,
then $\{c(v_{i}v_{t-1}),c(v_{i}v_{t})\}$ has at most one color not in $\{\alpha,\beta\}$ for any $2\leq i\leq t-2$.
In view of the path $v_{1}\rightarrow v_{t}\rightarrow v_{0}$ and the fact that $T^{*}$ has no rainbow $(v_{1},v_{0})$-path,
we have $c(v_{1}v_{t})=c(v_{t}v_{0})=\alpha$, see Figure \ref{figure: T*}.
So $\{c(v_{i}v_{t-1}),c(v_{i}v_{t})\}$ has at most one color not in $\{\alpha,\beta\}$ for any $1\leq i\leq t-2$.
Let $T''$ be the tournament obtained from $T^{*}$
by deleting $v_{t}$ and reversing the orientation of the arc $v_{0}v_{t-1}$, i.e., $v_{t-1}v_{0}\in A(T'')$.
Define the arc coloring of $T''$ as follows.
Color the arc $v_{t-1}v_{0}$ with $\alpha$,
color the arc $vv_{t-1}$ with $\gamma$ if $\{c(vv_{t-1}),c(vv_{t})\}$ has a color $\gamma\notin \{\alpha, \beta\}$ in $T^{*}$,
and color the arc $vv_{t-1}$ with $\alpha$ if $\{c(vv_{t-1}),c(vv_{t})\}\subseteq \{\alpha,\beta\}$ in $T^{*}$.
For other arcs $v'v''$,
if $c(v'v'')=\beta$ in $T^{*}$, then change its color to $\alpha$ in $T''$;
if $c(v'v'')\neq \beta$ in $T^{*}$, then remain its color in $T''$.
Note that each strongly connected $k$-vertex subtournament $T''_{k}$ of $T''$, $3\leq k\leq t$, is colored with at least $k-1$ colors.
If some $T''_{k}$ is colored with at most $k-2$ colors,
then the strongly connected $(k+1)$-vertex subtournament $T_{k+1}^{*}$ of $T^{*}$ induced by $V(T''_{k})\cup \{v_{t}\}$
is colored with at most $k-1$ colors, a contradiction.
Recall that $T^{*}$ is a minimum counterexample of the claim and $|V(T'')|<|V(T^{*})|$.
It therefore follows that $T''$ has a rainbow $(v_{i},v_{i-1})$-path $P$ for some $1\leq i\leq t-1$.
Note that $P$ must pass through the arc $v_{t-1}v_{0}\in A(T'')$ which is colored with $\alpha$
and thus no other arc of $P$ is colored with $\alpha$.
By the definition of the arc coloring of $T''$,
each arc in $P-v_{t-1}v_{0}$ is colored with some color not in $\{\alpha,\beta\}$ in the original tournament $T^{*}$.
If $i=t-1$,
then, since $c(v_{t-1}v_{t})=\beta$ in $T^{*}$,
we get that $v_{t-1}v_{t}v_{0}Pv_{t-2}$ is a rainbow $(v_{t-1},v_{t-2})$-path in $T^{*}$, a contradiction.
So $i\leq t-2$.
Let $w$ be the predecessor of $v_{t-1}$ on the path $P$.
Then $wv_{t-1}$ is colored with a color $\gamma\notin \{\alpha,\beta\}$.
This implies that either $c(wv_{t})=\gamma$ or $c(wv_{t-1})=\gamma$ in $T^{*}$.
So either $v_{i}Pwv_{t}v_{0}Pv_{i-1}$ or $v_{i}Pwv_{t-1}v_{t}v_{0}Pv_{i-1}$ is a rainbow $(v_{i},v_{i-1})$-path in $T^{*}$, a contradiction.

Let $x$ be the minimum integer such that
$c(v_{x}v_{t}),c(v_{x}v_{t-1}),\alpha,\beta$ are four distinct colors.
Recall that $c(v_{1}v_{t})=\alpha$.
So $x\geq 2$.
Assume w.l.o.g. that $c(v_{x}v_{t})=\alpha^{*}$ and $c(v_{x}v_{t-1})=\beta^{*}$.
Since there exists no rainbow $(v_{x},v_{x-1})$-path,
we have $c(v_{0}v_{x-1})=\alpha$, see Figure \ref{figure: T*};
otherwise, either $v_{x}\rightarrow v_{t}\rightarrow v_{0}\rightarrow v_{x-1}$ or $v_{x}\rightarrow v_{t-1}\rightarrow v_{t}\rightarrow v_{0}\rightarrow v_{x-1}$ is a rainbow $(v_{x},v_{x-1})$-path.
Since the triangle $(v_{0},v_{x-1},v_{t},v_{0})$ is colored with at least 2 colors and $c(v_{x-1}v_{t})\neq \beta$,
we have $c(v_{x-1}v_{t})\notin \{\alpha,\beta\}$.

Let $y$ be the minimum integer such that
$c(v_{0}v_{y})=\alpha$ and $c(v_{y}v_{t})\notin \{\alpha,\beta\}$.
Recall that $c(v_{1}v_{t})=\alpha$.
So $y\geq 2$.
Assume w.l.o.g. that $c(v_{y}v_{t})=\gamma$.
In view of the vertex $v_{x-1}$,
we get that $y$ exists and $y\leq x-1$.
By the definitions of $x$ and $y$,
for each $1\leq i\leq y-1$,
we have that $\{c(v_{i}v_{t-1}),c(v_{i}v_{t})\}$ has at most one color not in $\{\alpha,\beta\}$
and thus $\{c(v_{i}v_{y}),c(v_{i}v_{t-1}),c(v_{i}v_{t})\}$ has at most two colors not in $\{\alpha,\beta,\gamma\}$.

\begin{itemize}
  \item $\{c(v_{z}v_{y}),c(v_{z}v_{t-1}),c(v_{z}v_{t})\}$ has exactly two colors not in $\{\alpha,\beta,\gamma\}$ for some $1\leq z\leq y-1$.
\end{itemize}

Assume the opposite that $\{c(v_{i}v_{y}),c(v_{i}v_{t-1}),c(v_{i}v_{t})\}$ has at most one color not in $\{\alpha,\beta,\gamma\}$ for any $1\leq i\leq y-1$.
Then we consider the tournament $T'''$ obtained from $T^{*}$
by deleting the vertices in $\{v_{y+1},\ldots,v_{t}\}$ and reversing the orientation of the arc $v_{0}v_{y}$, i.e., $v_{y}v_{0}\in A(T''')$.
Define the arc coloring of $T'''$ as follows.
Color the arc $v_{y}v_{0}$ with $\alpha$,
color the arc $vv_{y}$ with $\eta$ if $\{c(vv_{y}),c(vv_{t-1}),c(vv_{t})\}$ has a color $\eta \notin \{\alpha,\beta,\gamma\}$ in $T^{*}$,
and color the arc $vv_{y}$ with $\alpha$ if $\{c(vv_{y}),c(vv_{t-1}),c(vv_{t})\}\subseteq \{\alpha,\beta,\gamma\}$ in $T^{*}$.
For other arcs $v'v''$,
if $c(v'v'')\in \{\alpha,\beta,\gamma\}$ in $T^{*}$, then change its color to $\alpha$ in $T'''$;
if $c(v'v'')\notin \{\alpha,\beta,\gamma\}$ in $T^{*}$, then remain its color in $T'''$.
Note that each strongly connected $k$-vertex subtournament $T'''_{k}$ of $T'''$, $3\leq k\leq y+1$, is colored with at least $k-1$ colors.
If some $T'''_{k}$ is colored with at most $k-2$ colors,
then the strongly connected $(k+2)$-vertex subtournament $T_{k+2}^{*}$ of $T^{*}$ induced by $V(T'''_{k})\cup \{v_{t-1},v_{t}\}$
is colored with at most $k$ colors, a contradiction.
By the minimality of $T^{*}$ and the fact $|V(T''')|<|V(T^{*})|$,
we get that $T'''$ has a rainbow $(v_{i},v_{i-1})$-path $P$ for some $1\leq i\leq y$.
Note that $P$ must pass through the arc $v_{y}v_{0}\in A(T''')$ which is colored with $\alpha$
and thus no other arc of $P$ is colored with $\alpha$.
By the definition of the arc coloring of $T'''$,
each arc in $P-v_{y}v_{0}$ is colored with some color not in $\{\alpha,\beta,\gamma\}$ in the original tournament $T^{*}$.
If $i=y$,
then, since $c(v_{y}v_{t})=\gamma$ in $T^{*}$,
we get that $v_{y}v_{t}v_{0}Pv_{y-1}$ is a rainbow $(v_{y},v_{y-1})$-path in $T^{*}$, a contradiction.
So $i\leq y-1$.
Let $w'$ be the predecessor of $y$ on the path $P$.
Then $w'v_{y}$ is colored with a color $\eta\notin \{\alpha,\beta,\gamma\}$ in $T'''$.
This implies that $c(w'v_{t})=\eta$ or $c(w'v_{t-1})=\eta$ or $c(w'v_{y})=\eta$ in $T^{*}$.
Recall that $c(v_{t-1}v_{t})=\beta$, $c(v_{y}v_{t})=\gamma$ and the colors $\alpha,\beta,\gamma,\eta$ are different.
So $v_{i}Pw'v_{t}v_{0}Pv_{i-1}$
or $v_{i}Pw'v_{t-1}v_{t}v_{0}Pv_{i-1}$
or $v_{i}Pw'v_{y}v_{t}v_{0}Pv_{i-1}$ is a rainbow $(v_{i},v_{i-1})$-path in $T^{*}$, a contradiction.

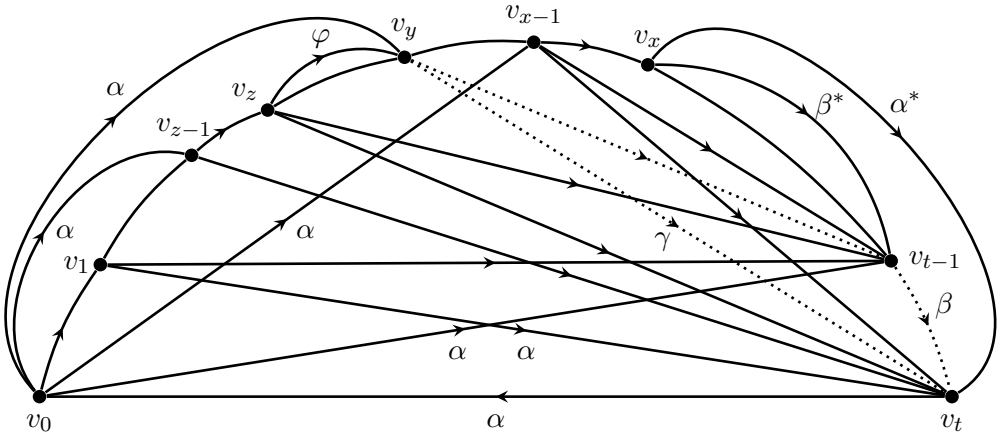
\begin{figure}[ht]
\begin{center}
\begin{tikzpicture}
\tikzstyle{vertex}=[circle,inner sep=1.8pt, minimum size=0.1pt]


\vertex (a)[fill] at (-6,0)[label=below:$v_{0}$]{};
\vertex (b)[fill] at (6,0)[label=below:$v_{t}$]{};
\vertex (c)[fill] at (5.2,1.8)[label=right:$v_{t-1}$]{};
\vertex (d)[fill] at (2,4.4)[label=above:$v_{x}$]{};
\vertex (e)[fill] at (0.5,4.7)[label=above:$v_{x-1}$]{};
\vertex (f)[fill] at (-1.2,4.5)[label=above:$v_{y}$]{};

\vertex (q)[fill] at (-5.2,1.75)[]{};
\vertex (g)[fill] at (-3,3.8)[]{};
\vertex (h)[fill] at (-4,3.2)[]{};

\vertex (i)[] at (-5.1,1.75)[label=left:$v_{1}$]{};
\vertex (i)[] at (-3.3,3.7)[label=above:$v_{z}$]{};
\vertex (i)[] at (-4.1,3.2)[label=above:$v_{z-1}$]{};

\vertex (i)[] at (0,0)[label=below:$\alpha$]{};
\vertex (i)[] at (-0.5,0.9)[label=below:$\alpha$]{};
\vertex (i)[] at (0.4,0.9)[label=below:$\alpha$]{};
\vertex (i)[] at (-2.5,2.5)[label=below:$\alpha$]{};
\vertex (i)[] at (-5.66,2.5)[label=below:$\alpha$]{};
\vertex (i)[] at (-5,4.4)[label=below:$\alpha$]{};
\vertex (i)[] at (5.39,4.28)[label=below:$\alpha^{*}$]{};
\vertex (i)[] at (4.4,4.25)[label=below:$\beta^{*}$]{};

\vertex (i)[] at (5.9,1.2)[label=center:$\beta$]{};
\vertex (i)[] at (2.2,2.41)[label=below:$\gamma$]{};
\vertex (i)[] at (-2.3,4.4)[label=above:$\varphi$]{};


\draw [directed,line width=1pt] (b)--(a);

\path (c) edge[dotted,directed,line width=1pt,bend left=9] (b);
\path (d) edge[line width=1pt,bend left=10] (c);
\path (e) edge[directed,line width=1pt,bend left=10] (d);

\path (a) edge[directed,line width=1pt] (c);
\path (e) edge[directed,line width=1pt] (b);
\path (e) edge[directed,line width=1pt] (c);

\path (a) edge[directed,line width=1pt] (e);

\path (a) edge[directed,line width=1pt,bend left=70] (h);
\path (g) edge[directed,line width=1pt,bend left=40] (f);

\path (d) edge[directed,line width=1pt,bend left=40] (c);
\path (d) edge[directed,line width=1pt,bend left=100] (b);
\path (a) edge[directed,line width=1pt,bend left=90] (f);

\path (f) edge[line width=1pt,bend left=10] (e);
\path (f) edge[dotted,directed,line width=1pt] (b);
\path (f) edge[dotted,directed,line width=1pt] (c);

\path (g) edge[line width=1pt,bend left=10] (f);
\path (h) edge[directed,line width=1pt,bend left=10] (g);

\path (g) edge[directed,line width=1pt] (b);
\path (g) edge[directed,line width=1pt] (c);

\path (h) edge[directed,line width=1pt] (b);

\path (a) edge[directed,line width=1pt,bend left=10] (q);
\path (q) edge[line width=1pt,bend left=10] (h);
\path (q) edge[directed,line width=1pt] (b);
\path (q) edge[directed,line width=1pt] (c);

\end{tikzpicture}
\end{center}
\caption{The arc-colored tournament $T^{*}$.}
\label{figure: T*}
\end{figure}

Now let $v_{z}$ be a vertex with $1\leq z\leq y-1$ such that $\{c(v_{z}v_{y}),c(v_{z}v_{t-1}),c(v_{z}v_{t})\}$ has exactly two colors not in $\{\alpha,\beta,\gamma\}$.
Assume that $\varphi,\omega$ are two colors not in $\{\alpha,\beta,\gamma\}$ and assume w.l.o.g. that $c(v_{z}v_{y})=\varphi$ and $\omega\in \{c(v_{z}v_{t-1}),c(v_{z}v_{t})\}$.
If $z=1$,
then $c(v_{1}v_{t-1})=\omega$ since $c(v_{1}v_{t})=\alpha$.
It follows that there exists a rainbow $(v_{1},v_{0})$-path $v_{1}\rightarrow v_{t-1}\rightarrow v_{t}\rightarrow v_{0}$, a contradiction.
So $z\geq 2$.
We will get a contradiction by showing that $c(v_{z-1}v_{0})=\alpha$ and $c(v_{z-1}v_{t})\notin \{\alpha,\beta\}$.
Recall that $T^{*}$ has no rainbow $(v_{z},v_{z-1})$-path
and the colors in $\{\alpha,\beta,\gamma,\varphi,\omega\}$ are different.
If $c(v_{z}v_{t})=\omega$,
then, in view of the two rainbow paths
$v_{z}\rightarrow v_{y}\rightarrow v_{t}\rightarrow v_{0}$
and $v_{z}\rightarrow v_{t}\rightarrow v_{0}$,
we have $c(v_{0}v_{z-1})=\alpha$.
If $c(v_{z}v_{t-1})=\omega$,
then, in view of the two rainbow paths
$v_{z}\rightarrow v_{y}\rightarrow v_{t}\rightarrow v_{0}$
and $v_{z}\rightarrow v_{t-1}\rightarrow v_{t}\rightarrow v_{0}$,
we have $c(v_{0}v_{z-1})=\alpha$.
Since the triangle $(v_{0},v_{z-1},v_{t},v_{0})$ is colored with at least 2 colors,
we have $c(v_{z-1}v_{t})\neq \alpha$.
Note that $c(v_{z-1}v_{t})\neq c(v_{t-1}v_{t})=\beta$.
So $c(v_{z-1}v_{t})\notin \{\alpha,\beta\}$.
The existence of $v_{z-1}$ contradicts the minimality of $y$.
The proof is complete.
\end{proof}

\section{Proof of Theorem \ref{thm: PKT is polynomial}}
\label{section: proof of thm 3}

Analogous to the definition of rainbow closure,
define the {\em PC closure} $\mathscr{C}_{_{PC}}(D)$ of an arc-colored digraph $D$
to be a digraph with vertex set $V(\mathscr{C}_{_{PC}}(D))=V(D)$
and arc set
$$
A(\mathscr{C}_{_{PC}}(D))=\{uv:~there~is~a~properly~colored~(u,v)\textrm{-}path~in~D\}.
$$
It is not difficult to see that the following simple (but useful) result holds, see also in \cite{BFZ2017}.

\begin{observation}[Bai. et al. \cite{BFZ2017}]\label{observation: pcp kernel}
An arc-colored digraph $D$ has a PCP-kernel if and only if its PC closure $\mathscr{C}_{_{PC}}(D)$ has a kernel.
\end{observation}

Let $T$ be an arbitrary $m$-colored $n$-vertex tournament
with vertex set $V(T)=\{v_{1},\ldots,v_{n}\}$
and color set $C(T)=\{c_{1},\ldots,c_{m}\}$.
Note that each possible PCP-kernel of $T$ consists of one vertex.
By Observation \ref{observation: pcp kernel},
if we get the closure $\mathscr{C}_{_{PC}}(T)$ of $T$
then we only need to check whether or not there exists a vertex with indegree $n-1$,
which can be done in $O(n)$ time.
It therefore suffices to show that we can construct the closure $\mathscr{C}_{_{PC}}(T)$ in polynomial time.
We first construct a digraph $D_{_{T}}$ with respect to $T$ as follows.
Let
$$
V(D_{_{T}})=\{(v_{i},c_{j}):1\leq i\leq n, 1\leq j\leq m\}.
$$
Note that $D_{_{T}}$ has $nm$ vertices.
We will construct a sequence $D_{0},D_{1},\ldots,D_{n-1}=D_{_{T}}$ of digraphs
such that, for each $0\leq i\leq n-1$,
it holds that $((v,c),(v',c'))\in D_{i}$ if and only if
$T$ has a properly colored $(v,v')$-path of length at most $i$ and
satisfying that the first arc is colored with $c$ and the last arc is colored with $c'$.
The starting point for our construction is an empty digraph $D_{0}$ with vertex set $V(D_{0})=V(D_{_{T}})$.
The main ideas are as follows.

\begin{itemize}
  \item For each arc $v'v''$ with $c(v'v'')=c$ in $T$,
  add an arc from $(v',c)$ to $(v'',c)$ in $D_{0}$,
  denote the resulting digraph by $D_{1}$.
  \item Assume that we have constructed $D_{k-1}$,
  where $k\geq 2$ is an integer.
  For each pair of vertices $(v',c'),(v'',c'')$ with $((v',c'),(v'',c''))\notin D_{k-1}$,
  add an arc from $(v',c')$ to $(v'',c'')$ in $D_{k}$
  if there exists a vertex $(v,c)$ such that
  $((v',c'),(v,c))\in D_{k-1}$, $vv''\in T$, $c(vv'')=c''$ and $c\neq c''$.
\end{itemize}

We continue the above process to get $D_{n-1}$.
Now we construct the closure $\mathscr{C}_{_{PC}}(T)$ of $T$.
First let $V(\mathscr{C}_{_{PC}}(T))=V(T)$.
Then for each pair of vertices $v',v''$ in $\mathscr{C}_{_{PC}}(T)$,
we add an arc from $v'$ to $v''$
whenever there exists an arc from $(v',c')$ to $(v'',c'')$ in $D_{n-1}$ for some $c',c''\in C(T)$.
Here note that if $((v',c'),(v'',c''))\in D_{n-1}$ then there exists a properly colored $(v',v'')$-path in $T$.
One can see that the above process can be terminated in $O(n^{3}m^{3})$ time.

\section{Concluding remarks}
\label{section: concluding}

It is not difficult to see that $RPOG\in NP$.
Combining the NP-hardness of RPOG in the proof of Theorem \ref{thm: RKT is NPC},
we have the following result.

\begin{corollary}
The RPOG problem is NP-complete.
\end{corollary}

Similar to the proof of Theorem \ref{thm: rainbow kernels in tournaments},
we can get that an arc-colored tournament with all triangles rainbow has a rainbow kernel.
Here we give the sketch of the proof.

\begin{corollary}\label{corollary: triangles are rainbow}
Every arc-colored tournament with all triangles rainbow has a rainbow kernel.
\end{corollary}

\begin{proof}[Sketch of proof]
Assume the opposite that the result does not hold and consider a counterexample $T$ with minimum order $n$.
Similar to the proof in Theorem \ref{thm: rainbow kernels in tournaments},
we can get that $T$ has a Hamilton cycle $(v_{0},v_{1},\ldots,v_{n-1},v_{0})$ such that
$T$ has no rainbow $(v_{i},v_{i-1})$-path for any $0\leq i\leq n-1$,
where $V(T)=\{v_{0},v_{1},\ldots,v_{n-1}\}$ and the subtraction is modulo $n$.
Let $t$ be the maximum integer such that $v_{0}\rightarrow \{v_{1},\ldots,v_{t-1}\}$ and $v_{t}\rightarrow v_{0}$.
Such an integer $t$ exists as $v_{n-1}\rightarrow v_{0}$.
Now $(v_{0},v_{t-1},v_{t},v_{0})$ is a triangle and by assumption it is rainbow.
This implies that $T$ has a rainbow $(v_{t},v_{t-1})$-path $v_{t}\rightarrow v_{0}\rightarrow v_{t-1}$, a contradiction.
\end{proof}

Inspired by Observation \ref{observation: rainbow kernel in acyclic digraphs} and Corollary \ref{corollary: triangles are rainbow},
we conjecture that `all cycles rainbow' may be a sufficient condition for the existence of a rainbow kernel in general digraphs,
however, no other strong evidence has been found till now,
so we only state our initial idea as a problem here.

\begin{problem}
Is it true that every arc-colored digraph with all cycles rainbow has a rainbow kernel?
\end{problem}

\end{spacing}
\end{document}